\documentclass[12pt,reqno]{amsart}
\usepackage{fullpage,amssymb,amsmath,mathrsfs,enumerate}
\newtheorem{theorem}{Theorem}[section]
\newtheorem{lemma}[theorem]{Lemma}
\newtheorem{proposition}[theorem]{Proposition}
\newtheorem{corollary}[theorem]{Corollary}
\newtheorem{fact}[theorem]{Fact}

\theoremstyle{definition}
\newtheorem*{definition}{Definition}
\newtheorem{convention}[theorem]{Convention}

\theoremstyle{remark}
\newtheorem{claim}{Claim}

\newcommand\setsep{;\ }
\def\er{\mathbb R}

\def\O{\mathcal{O}}
\def\N{\mathcal{N}}
\def\F{\mathcal F}
\def\B{\mathcal B}
\def\en{\mathbb N}
\def\C{\mathcal{C}}
\def\ov{\overline}
\def \exp {\operatorname{exp}}
\def \rng {\operatorname{Rng}}
\def \dom {\operatorname{Dom}}
\def \sspan {\operatorname{span}}

\begin{document}
\title{Monotone retractability and retractional skeletons}
\author{Marek C\'uth and Ond\v{r}ej F.K. Kalenda}
\address{Department of Mathematical Analysis \\
Faculty of Mathematics and Physic\\ Charles University\\
Sokolovsk\'{a} 83, 186 \ 75\\Praha 8, Czech Republic}
\email{marek.cuth@gmail.com}
\email{kalenda@karlin.mff.cuni.cz}
\thanks{Our research was supported by the grant GA \v{C}R P201/12/0290.}
\begin{abstract}We prove that a  countably compact space is monotonically retractable if and only if it has a full retractional skeleton.
In particular, a compact space is monotonically retractable if and only if it is Corson. 
This gives an answer to a question of R. Rojas-Hern{\'a}ndez and V. V. Tkachuk.
Further, we apply this result to characterize retractional skeleton using a topology on the space of continuous functions, answering thus a question  of the first author and a related question of W. Kubi\'s.\end{abstract}
\maketitle

\section{Introduction and main results}
Spaces with a rich family of retractions often occur both in topology and functional analysis.
For example, systems of retractions were used by  Amir and Lindenstrauss to characterize Eberlein compact spaces \cite{amirLind} or by Gul$'$ko \cite{gul} to prove that a compact space $K$ is Corson whenever $\C_p(K)$ has the Lindel\"of $\Sigma$-property. This line of research
continued for a long time (for a survey see e.g. \cite{kalendaSurvey} and Chapter 19 of \cite{kubisKniha}). The optimal notion of an indexed system
of retractions in this area was defined in \cite{kubisSmall}. We slightly generalize this notion to a more general situation -- we consider countably compact spaces, not only compact ones.

\begin{definition}A \textit{retractional skeleton} in a countably compact space $X$ is a family of continuous retractions $\mathfrak{s} = \{r_s\}_{s\in\Gamma}$, indexed by an up-directed partially ordered set $\Gamma$, such that
\begin{enumerate}[\upshape (i)]
	\item $r_s[X]$ is a metrizable compact for each $s\in\Gamma$,
	\item $s,t\in\Gamma$, $s\leq t \Rightarrow r_s = r_s\circ r_t = r_t\circ r_s$,
	\item given $s_0 < s_1 < \cdots$ in $\Gamma$, $t = \sup_{n\in\omega}s_n$ exists and $r_t(x) = \lim_{n\to\infty}r_{s_n}(x)$ for every $x\in X$,
	\item for every $x\in X$, $x = \lim_{s\in\Gamma}r_s(x)$.
\end{enumerate}
We say that $D(\mathfrak{s}) = \bigcup_{s\in\Gamma}r_s[X]$ is the \textit{set induced by the retractional skeleton} $\mathfrak{s}$ in $X$.\\
If $D(\mathfrak{s}) = X$ then we say that $\mathfrak{s}$ is a \textit{full retractional skeleton}.
\end{definition}

Let us point out that the condition (i) in the definition of a retractional skeleton is equivalent to
\begin{itemize}
	\item[(i')] $r_s[X]$ has a countable network for each $s\in\Gamma$.
\end{itemize}
Indeed, any metrizable compact has a countable network. Conversely, if $r$ is any retraction on $X$, then $r[X]$ is a closed subset of $X$, hence it is countably compact. If it has countable network, it is Lindel\"of  and hence compact. Finally, a compact with countable network is metrizable.

Another type of a structured system of retractions was recently introduced in \cite{roj14}. 

\begin{definition}
A space $X$ is \emph{monotonically retractable} if we can assign to any countable set $A\subset X$ a retraction $r_A$ and a countable set $\N(A)$
such that the following conditions are fulfilled:
\begin{itemize}
	\item $A\subset r_A[X]$.
	\item The assignment $A\mapsto \N(A)$ is $\omega$-monotone, i.e.,
	\begin{itemize}
  	\item[(i)] if $A\subset B$ are countable subsets of $X$, then $\N(A)\subset\N(B)$;
  	\item[(ii)] if $(A_n)$ is an increasing sequence of countable subsets of $X$, then $\N(\bigcup_{n=1}^\infty A_n)=\bigcup_{n=1}^\infty \N(A_n)$.
 \end{itemize}
  \item $\N(A)$ is a network of $r_A$, i.e. $r_A^{-1}(U)$ is the union of a subfamily of $\N(A)$ for any open set $U\subset X$.
\end{itemize}
\end{definition}

Our main result is the following characterization of monotonically retractable countably compact spaces using the notion of a retractional skeleton.

\begin{theorem}\label{t:charactFullSkeleton}
A countably compact space is monotonically retractable if and only if it has a full retractional skeleton.
\end{theorem}

Since a compact space has a full retractional skeleton if and only if it is Corson \cite[Theorem 3.11]{cuthCMUC}, we get 
the following positive answer to Question 6.1 of \cite{tkaRoj}.

\begin{corollary}\label{t:charactCorson}
A compact space is monotonically retractable if and only if it is Corson.
\end{corollary}

As another corollary we obtain the following result from \cite{roj14}.

\begin{corollary}\label{c:roj14}
Any first countable countably compact subspace of an ordinal is monotonically retractable.
\end{corollary}

Indeed, it is enough to observe that any first countable countably compact subspace of an ordinal admits a full retractional skeleton.
To do that one can use the formula from \cite[Example 6.4]{kubisSmall}.

Theorem~\ref{t:charactFullSkeleton} will be proved in Section 4, it is an immediate consequence of Theorem~\ref{t:main}. Corollary~\ref{t:charactCorson}
is in fact easier, it follows already from Proposition~\ref{p:charactCorson}.

Further, we apply our results to prove the following `noncommutative' analogues of the results of \cite{kalendaCharact}. These theorems provide answers to Problem 1 of \cite{cuthCMUC} and Problem 1 of \cite{kubisSkeleton}. The topological property sought in the quoted problems is `to be monotonically Sokolov'. This class of spaces was introduced and studied in \cite{tkaRoj}, we recall the definition in the next section. 

\begin{theorem}\label{t:answerCuth}Let $K$ be a compact space and $D$ be a dense subset of $K$. Then the following two conditions are equivalent:
	\begin{enumerate}[\upshape (i)]
		\item $D$ is induced by a retractional skeleton in $K$.
		\item $D$ is countably compact and $(\C(K),\tau_p(D))$ is monotonically Sokolov.
	\end{enumerate}
\end{theorem}

This theorem will be proved in the last section. The next one is its Banach-space counterpart. Projectional skeleton is Banach-space analogue of retractional skeleton, these notions are dual in a sense. For exact definitions and details see \cite{kubisSkeleton} or \cite{cuthSimul}. 

\begin{theorem}\label{t:answerKubis}Let $E$ be a Banach space and $D\subset E^*$ a norming subspace. Then the following two conditions are equivalent:
	\begin{enumerate}[\upshape (i)]
		\item $D$ is induced by a projectional skeleton.
		\item $D$ is weak$^*$ countably closed and $(E,\sigma(E,D))$ is monotonically Sokolov.		
	\end{enumerate}
\end{theorem}

This theorem will proved in the last section, using its more precise version, Theorem~\ref{t:skeletonIffXSokolov}. 

\section{Preliminaries}

In this section we collect basic notation, terminology and some known facts which will be used in the sequel.

We denote by $\omega$ the set of all natural numbers (including $0$), by $\en$ the set $\omega\setminus\{0\}$. If $X$ is a set then $\exp(X) = \{Y\setsep Y\subset X\}$. We denote by $[X]^{\leq\omega}$ all countable subsets of $X$.

All topological spaces are assumed to be Tychonoff. Let $T$ be a topological space.
\begin{itemize}
	\item  $\tau(T)$ denotes the topology of $T$ and $\tau(x,T) = \{U\in\tau(T)\setsep x\in U\}$ for any $x\in T$.
	\item  A subset $S\subset T$ is said to be countably closed if $\ov{C}\subset S$ for every countable subset $C\subset S$. It is easy to check that a countably closed subset of a countably compact space is countably compact.
  \item   A family $\N$ of subsets of $T$ is said to be a \emph{network} of $T$ if any open set in $T$ is the union of a subfamily of $\N$
	\item If $A\subset T$, then a family $\N$ of subsets of $T$ is said to be an \emph{external network} of $A$ in $T$ if for any $a\in A$ and $U\in\tau(a,T)$ there exists $N\in\N$ such that $a\in N\subset U$.
	\item If $Y$ is a topological space and $f:T\to Y$ is a continuous map, then a family $\N$ of subsets of $T$ is said to be a \emph{network of $f$} if for any $x\in T$ and $U\in\tau(f(x),Y)$ there exists $N\in\N$ such that $x\in N$ and $f[N]\subset U$.
	\item $\beta T$ denotes the \v{C}ech-Stone compactification of $T$.
\end{itemize}

For any topological spaces $X$ and $Y$ the set of continuous functions from $X$ to $Y$ is denoted by $\C(X,Y)$. We write $\C(X)$ instead of $\C(X,\er)$ and $\C_b(X)$ for the set of all bounded functions from $\C(X)$. By $C_p(X)$ we denote the space $\C(X)$ equipped with the the topology of pointwise convergence (i.e., the topology inherited from $\er^X$). Moreover, if $D\subset X$ is dense, we denote by $\tau_p(D)$ the topology of the pointwise convergence on $D$ (i.e. the weakest topology on $\C(X)$ such that $f\mapsto f(d)$ is continuous for every $d\in D$).

We shall consider Banach spaces over the field of real numbers. If $E$ is a Banach space and $A\subset E$, we denote by $\sspan{A}$ the linear hull of $A$. $B_E$ is the closed unit ball of $E$; i.e., the set $\{x\in E:\; \|x\| \leq 1\}$. $E^*$ stands for the (continuous) dual space of $E$. For a set $A\subset E^*$ we denote by $\ov{A}^{w^*}$ the weak$^*$ closure of $A$. Given a set $D\subset E^*$ we denote by $\sigma(E,D)$ the weakest topology on $E$ such that each functional from $D$ is continuous. A set $D\subset E^*$ is \textit{r-norming} if $\|x\| \leq r. \sup\{|x^*(x)|:\;x^*\in D\cap B_{E^*}\}$. We say that a set $D\subset E^*$ is norming if it is $r$-norming for some $r\geq 1$.

The following definitions come from \cite{tkaRoj}.

\begin{definition}Let $X, Y$ be sets, $\O\subset\exp(X)$ closed under countable increasing unions, $\N\subset\exp(Y)$ and $f:\O\to\N$. We say that $f$ is \emph{$\omega$-monotone} if
\begin{enumerate}[\upshape (i)]
	\item $f(A)$ is countable for every countable $A\in\O$;
	\item if $A\subset B$ and $A,B\in\O$ then $f(A)\subset f(B)$;
	\item if $\{A_n\setsep n\in\omega\}\subset \O$ and $A_n\subset A_{n+1}$ for every $n\in\omega$ then $f(\bigcup_{n\in\omega}A_n) = \bigcup_{n\in\omega}f(A_n)$.
\end{enumerate}
\end{definition}

\begin{definition}A space $T$ is \emph{monotonically Sokolov} if we can assign to any countable family $\F$ of closed subsets of $T$ a continuous retraction $r_\F:T\to T$ and a countable external network $\N(\F)$ for $r_\F(T)$ in $T$ such that $r_\F(F)\subset F$ for every $F\in\F$ and the assignment $\N$ is $\omega$-monotone.
\end{definition}

In the following statement we sum up some properties of monotonically retractable and monotonically Sokolov spaces which we will use later.
They follow from results of \cite{tkaRoj}.

\begin{fact}\label{f:mrs}
Let $X$ be a topological space. Then:
	\begin{enumerate}[\upshape (a)]
		\item $X$ is monotonically retractable if and only if $\C_p(X)$ is monotonically Sokolov.
		\item $X$ is monotonically Sokolov if and only if $\C_p(X)$ is monotonically retractable.
		\item Any closed subspace of a monotonically retractable (resp. monotonically Sokolov) space is monotonically retractable (resp. monotonically Sokolov).
		\item A countable product of monotonically Sokolov spaces is monotonically Sokolov.
		\item Any monotonically Sokolov space is Lindel\"of.
		\item Any monotonically retractable space is normal and $\omega$-monolithic (i.e., any separable subset has countable network).
		\item Any monotonically retractable space has countable tightness.
		\item Any countably compact subset of a monotonically retractable space is closed (and hence monotonically retractable).
	\end{enumerate}
\end{fact}

\begin{proof} 
The assertions (a) and (b) are proved in \cite[Theorem 3.5]{tkaRoj}, the assertions (c)--(f) follow from \cite[Theorems 3.4 and 3.6]{tkaRoj}.

Let us show (g). Let $X$ be a monotonically retractable space. Consequently, $Y = \C_p(X)$ is monotonically Sokolov by (a) and by (d) $Y^n$ is monotonically Sokolov for every $n\in\en$. It follows by (e) that $Y^n$ is Lindel\"of for every $n\in\en$ and, by \cite[Theorem II.1.1]{archangelskii}, $\C_p(Y) = \C_p(\C_p(X))$ has a countable tightness. Since $X$ embeds in $\C_p(\C_p(X))$, it must have a countable tightness.

Finally, let us prove (h). Let $X$ be a monotonically retractable space and $A\subset X$ countably compact. Fix $a\in \ov{A}$. Since $X$ has a countable tightness (by (g)), there is a countable set $S\subset A$ with $a\in \ov{S}$. By (f) $\ov{S}\cap A$ has a countable network, hence it is Lindel\"of. Since $A$ is countably compact, $\ov{S}\cap A$ is countably compact and Lindel\"of; hence, compact. It follows that $\ov{S}\cap A$ is closed in $X$ and $a\in A$.
\end{proof}

In the following we summarize some easy facts concerning sets induced by a retractional skeleton.

\begin{fact}\label{f:basics}Let $X$ be a countably compact space and let $D$ be a set induced by 
a retractional skeleton in $X$. Then:
\begin{enumerate}[\upshape (i)]
	\item $D$ is countably closed in $X$.
	\item $D$ is sequentially compact.
	\item If $X$ is compact, then $X=\beta D$.
\end{enumerate}
\end{fact}
\begin{proof}Let $\mathfrak{s} = \{r_s\}_{s\in\Gamma}$ be a retractional skeleton in $X$ with $D = D(\mathfrak{s})$. Whenever $A\subset D$ is countable, there is $s\in\Gamma$ with $A\subset r_s[X]$; hence, $\ov{A}\subset r_s[X]$ is metrizable compact and (i) and (ii) follows. 
The assertion (iii) is proved in \cite[Theorem 32]{kubisSkeleton}.
\end{proof}

\section{The method of elementary models}

The purpose of this section is to briefly recall the reader of some basic facts concerning the method of elementary models. This is a set-theoretical method which can be used in various branches of mathematics. A. Dow in \cite{dow} illustrated the use of this method in topology, P. Koszmider in \cite{kos05} used it in functional analysis. Later, inspired by \cite{kos05}, W. Kubi\'s in \cite{kubisSkeleton} used it to construct retractional (resp. projectional) skeleton in certain compact (resp. Banach) spaces. In \cite{cuth} the method has been slightly simplified and specified. We briefly recall some basic facts. More details may be found e.g. in \cite{cuth} and \cite{cuthKalenda}.

First, let us recall some definitions. Let $N$ be a fixed set and $\phi$ a formula in the language of $ZFC$. Then the {\em relativization of $\phi$ to $N$} is the formula $\phi^N$ which is obtained from $\phi$ by replacing each quantifier of the form ``$\forall x$'' by ``$\forall x\in N$'' and each quantifier of the form ``$\exists x$'' by ``$\exists x\in N$''.

If $\phi(x_1,\ldots,x_n)$ is a formula with all free variables shown (i.e., a formula whose free variables are exactly $x_1,\ldots,x_n$) then  $\phi$ is said to be {\em absolute for $N$} if
\[
\forall a_1,\ldots,a_n\in N\quad (\phi^N(a_1,\ldots,a_n) \leftrightarrow \phi(a_1,\ldots,a_n)).
\]

A list of formulas, $\phi_1,\ldots,\phi_n$, is said to be {\em subformula closed} if  every subformula of a formula in the list is also contained in the list.

The method is based mainly on the following theorem (a proof can be found in \cite[Chapter IV, Theorem 7.8]{Kunen}).

\begin{theorem}\label{T:countable-model}
Let $\phi_1, \ldots, \phi_n$ be any formulas and $Y$ any set. Then there exists a set $M \supset Y$ such that
$\phi_1, \ldots, \phi_n \text{ are absolute for } M$ and $|M| \leq \max(\omega,|Y|)$.
\end{theorem}

Since the set from Theorem~\ref{T:countable-model} will often be used, the following notation is useful.

\begin{definition}
Let $\phi_1, \ldots, \phi_n$ be any formulas and $Y$ be any countable set.
Let $M \supset X$ be a countable set such that $\phi_1, \ldots, \phi_n$ are absolute for $M$.
Then we say that $M$ is an \emph{elementary model for $\phi_1,\ldots,\phi_n$ containing $X$}.
This is denoted by $M \prec (\phi_1,\ldots,\phi_n; Y)$.
\end{definition}

The fact that certain formula is absolute for $M$ will always be used in order to satisfy the assumption of the following lemma from \cite[Lemma 2.3]{cuthRmoutilZeleny}. Using this lemma we can force the model $M$ to contain all the needed objects created (uniquely) from elements of $M$.

\begin{lemma}\label{l:unique-M}
Let $\phi(y,x_1,\ldots,x_n)$ be a formula with all free variables shown and $Y$ be a countable set.
Let $M$ be a fixed set, $M \prec (\phi, \exists y \colon \phi(y,x_1,\ldots,x_n);\; Y)$, and
$a_1,\ldots,a_n \in M$ be such that there exists a set $u$ satisfying
$\phi(u,a_1,\ldots,a_n)$. Then there exists $u \in M$ such that $\phi(u,a_1,\ldots,a_n)$.
\end{lemma}

\begin{proof}Let us give here the proof just for the sake of completeness. Using the absoluteness of the formula $\exists u\colon \phi(u,x_1,\ldots,x_n)$ there exists $u\in M$ satisfying $\phi^M(u,a_1,\ldots,a_n)$.
Using the absoluteness of $\phi$ we get, that for this $u\in M$ the formula $\phi(u,a_1,\ldots,a_n)$ holds.
\end{proof}

It would be very laborious and pointless to use only the basic language of the set theory.
For example, having a function $f$, we often write $y = f(x)$ and we know that this is a shortcut for a formula with free variables $x$, $y$, and $f$.

Indeed, consider the formula
\[
\varphi(x,y,z) = \forall a (a\in z \leftrightarrow (a=x\vee a=y)).
\]
Then $\varphi(x,y,z)$ is true if and only if $z = \{x,y\}$. Recall that $y = f(x)$ means $\{\{x\},\{x,y\}\}\in f$. Hence, $y = f(x)$ if and only if the following formula is true
\[
\forall z (\forall a (a\in z \leftrightarrow \varphi(x,x,a)\vee \varphi(x,y,a)) \Rightarrow z\in f).
\]

Therefore, in the following text we use this extended language of the set theory as we are used to.
We shall also use the following convention.

\begin{convention}
Whenever we say ``\emph{for any suitable model $M$ (the following holds \dots)}''
we mean that  ``\emph{there exists a list of formulas $\phi_1,\ldots,\phi_n$ and a countable set $Y$ such that for every $M \prec (\phi_1,\ldots,\phi_n;Y)$ (the following holds \dots)}''.
\end{convention}

By using this new terminology we lose the information about the formulas $\phi_1,\ldots,\phi_n$ and the set $Y$.
However, this is not important in applications.

Let us recall several further results about elementary models
(all the proofs are based on Lemma \ref{l:unique-M} and they can be found in \cite[Proposition 2.9, 2.10 and 3.2]{cuth} and \cite[Lemma 4.8]{cuthCMUC}).

\begin{lemma}\label{l:predp}
There are formulas $\theta_1,\dots,\theta_m$ and a countable set $Y_0$ such that any $M\prec(\theta_1,\ldots,\theta_n;\; Y_0)$ satisfies the following conditions:
\begin{itemize}
	\item If $f\in M$ is a mapping, then $\dom(f)\in M$, $\rng(f)\in M$ and $f[M]\subset M$.
	\item If $A$ is finite, then $A\in M$ if and only if $A\subset M$.
	\item If $A\in M$ is a countable set, then $A\subset M$.
	\item If $A,B\in M$, then $A\cup B\in M$.
\end{itemize}
\end{lemma}

Moreover, we will need to find suitable models in a ``monotonic way''. Thus, the following lemma from \cite[Lemma 4]{cuthKalenda} will be useful as well.

\begin{lemma}\label{l:BasicSkolem}
 Let $\phi_1,\ldots,\phi_n$ be a subformula closed list of formulas and let $R$ be a set such that $\phi_1,\ldots,\phi_n$ are absolute for $R$. Then there exists a function $\psi:[R]^{\leq\omega}\to [R]^{\leq\omega}$ such that
 \begin{enumerate}[\upshape (i)]
	\item For every $A\in[R]^{\leq\omega}$, $\psi(A)\prec (\phi_1,\ldots,\phi_n; A)$.
	\item The mapping $\psi$ is $\omega$-monotone.
\end{enumerate}
\end{lemma}

\begin{definition}We say that the function from Lemma \ref{l:BasicSkolem} is a \emph{Skolem function for $\phi_1,\ldots,\phi_n$ and $R$}.
\end{definition}

\section{Proof of the main result}

In this section we are going to prove our main result, Theorem~\ref{t:charactFullSkeleton}. This is the content of the equivalence (i)$\Leftrightarrow$(iii) from the following theorem. We add one more equivalent condition, formulated with the use of elementary models, because via this
condition the proof will be done.

\begin{theorem}\label{t:main}Let $X$ be a countably compact space. Then the following are equivalent:
	\begin{itemize}
	\item[(i)] $X$ has a full retractional skeleton.
	\item[(ii)] For any suitable model $M$, $\C(X)\cap M$ separates the points of $\ov{X\cap M}$.
	\item[(iii)] $X$ is monotonically retractable.
\end{itemize}
\end{theorem}

Recall that a compact space is Corson if and only if it has a full retractional skeleton; see e. g. \cite[Theorem 3.11]{cuthCMUC}. 
Hence, in case $X$ is compact, the equivalence (i)$\Leftrightarrow$(ii) comes from \cite[Theorem 7]{ban91} (for a simplified and generalized version see \cite[Theorem 30]{kubisSkeleton}, a more detailed proof which suits our situation the most can be found in \cite[Theorem 4.9]{cuthCMUC}).

The rest of this section will be devoted to the proof of Theorem~\ref{t:main}.
We start by proving the implication (iii)$\Rightarrow$(ii). This is the content of the following proposition -- note that it holds even without
assuming countable compactness of $X$. Let us also remark that this already provides a proof of Corollary~\ref{t:charactCorson}.

\begin{proposition}\label{p:charactCorson} Let $X$ be a monotonically retractable space. Then, for any suitable model $M$, $\C_b(X)\cap M$ separates the points of $\ov{X\cap M}$.
\end{proposition}

\begin{proof}
Suppose that for any countable set $A\subset X$, we have a retraction $r_A:X\to X$ and a family $\N(A)$ that witness the monotone retractability of $X$. Fix formulas $\phi_1,\ldots,\phi_n$ containing the formulas $\theta_1,\dots,\theta_m$ from Lemma \ref{l:predp} and the formula (and its subformulas) marked by $(*)$ in the proof below, and a countable set $Y$ containing the set $Y_0$ from Lemma \ref{l:predp} and the set $\{\C_b(X),\N\}$. Fix $M\prec(\phi_1,...,\phi_n;\; Y)$. Put $A = \bigcup\{B\in[X]^{\leq\omega}\setsep B\in M\}$. This is a countable set.
\begin{claim}\label{claim1}$\N(A)\subset M$\end{claim}
\begin{proof}By Lemma \ref{l:predp}, the set $\{B\in[X]^{\leq\omega}\setsep B\in M\}$ is closed under finite unions. Hence, there exists an increasing sequence $(B_n)_{n\in\en}$ with $A = \bigcup_{n\in\en} B_n$ and $B_n\in M$ for every $n\in\en$. Since the assignment $\N$ is $\omega$-monotone, we have $\N(A) = \bigcup_{n\in\en}\N(B_n)$. Fix $n\in\en$. By Lemma~\ref{l:predp}, $\N(B_n)\in M$ and $\N(B_n)\subset M$. Consequently, $\N(A)\subset M$.\end{proof}
\begin{claim}\label{claim2}$\ov{X\cap M}\subset r_A[X]$\end{claim}
\begin{proof}Fix $x\in X\cap M$. Then $\{x\}\in [X]^{\leq\omega}\cap M$ (by Lemma~\ref{l:predp}); hence, $x \in A$. Thus, $X\cap M\subset A$ and $\ov{X\cap M}\subset \ov{A}\subset r_A[X]$.\end{proof}

Fix $x,y\in\ov{X\cap M}$, $x\neq y$. By Claim \ref{claim2}, $x,y\in r_A[X]$. Find sets $U\in\tau(x,r_A[X])$ and $V\in\tau(y,r_A[X])$ such that $\ov{U}\cap\ov{V} = \emptyset$. Since $\N(A)$ is a network of $r_A$, we can find $N_x\in\N(A)$ and $N_y\in\N(A)$ with $x\in N_x\subset r_A^{-1}[U]$ and $y\in N_y\subset r_A^{-1}[V]$. Note that $\ov{N_x}\cap\ov{N_y} = \emptyset$ and recall that $X$ is normal by Fact~\ref{f:mrs}. By Claim \ref{claim1}, $N_x, N_y\in M$. Hence, by Lemma \ref{l:unique-M} and the absoluteness of the formula (and its subformula)
\[
\exists f\in\C_b(X)\quad(\forall a\in N_x: f(a) = 0 \wedge \forall b\in N_y: f(b) = 1),\eqno{(*)}
\]
there is $f\in\C_b(X)\cap M$ with $f(x) = 0 \neq 1 = f(y)$. Thus, $\C_b(X)\cap M$ separates the points of $\ov{X\cap M}$.
\end{proof}

We continue with proving the equivalence (i)$\Leftrightarrow$(ii) from Theorem~\ref{t:main}. In fact, this equivalence
is essentially known due to the following result:

\begin{lemma}\label{l:generskel} {\rm (\cite[Theorem 30]{kubisSkeleton}, see also \cite[Theorem 4.9]{cuthCMUC})}
Let $K$ be a compact space and $X\subset K$ a dense countably compact subset. The following assertions are
equivalent:
\begin{itemize}
	\item[(i)] $X$ is induced by a retractional skeleton in $K$.
	\item[(ii)] For any suitable model $M$, $\C(K)\cap M$ separates the points of $\ov{X\cap M}$.
\end{itemize}
\end{lemma}

In fact, the quoted results use a slightly stronger assumption that $X$ is countably closed in $K$. However, if $X$ is induced by a retractional skeleton, it is automatically countably closed by Fact~\ref{f:basics}, so this assumption is not used for (i)$\Rightarrow$(ii).
For the opposite implication, by \cite[Theorem 4.9]{cuthCMUC} it follows from (ii) that $X$ is contained in a set $Y$ induced by retractional skeleton.
Now it follows easily from \cite[Theorem 32]{kubisSkeleton} that $X=Y$. 

The proof of the equivalence (i)$\Leftrightarrow$(ii) from Theorem~\ref{t:main} will be done by reducing the situation to the use
of Lemma~\ref{l:generskel}. More precisely, let us consider $K=\beta X$. We will show that the validity of assertion (i) in Theorem~\ref{t:main}
is equivalent to the validity of (i) in Lemma~\ref{l:generskel} and similarly for the respective assertions (ii).
We begin with the assertions (ii). The key tool to do that is the following easy lemma.

\begin{lemma}\label{l:separatePointsInCompactification} Let $L$ be a compact space and $A\subset L$ a dense countably compact subset. Let $S\subset \C(L)$ be a countable set separating the points of $A$. Then $S$ separates the points of $L$.
\end{lemma}
\begin{proof}Arguing by contradiction, let $x_1,x_2\in L$ be such that $x_1\neq x_2$ and $f(x_1) = f(x_2)$ for every $f\in S$. Find $g\in\C(L)$ with $g(x_1)\neq g(x_2)$. Denote, for $i\in\{1,2\}$,
\[
A_i = \bigcap_{h\in S\cup\{g\}}\{t\in L\setsep h(x_i) = h(t)\}.
\]
Then $A_1, A_2$ are nonempty $G_\delta$ sets and $A_1\cap A_2 = \emptyset$. Hence, e. g. by \cite[Lemma 1.11]{kalendaSurvey}, there are $y_1\in A\cap A_1$ and $y_2\in A\cap A_2$. This is a contradiction because $S$ does not separate the points $y_1\neq y_2$ from $A$.
\end{proof}

Let us now show the equivalence of the respective assertions (ii) from Theorem~\ref{t:main} and Lemma~\ref{l:generskel}. 
Suppose that the assertion (ii) from Theorem~\ref{t:main} holds. Let $M$ be such a suitable model containing moreover the
extension map $f\mapsto \beta f$, $f\in \C(X)$, and the restriction map $f\mapsto f\restriction_X$, $f\in\beta X$. Then
$\C(X)\cap M=\{f\restriction_X:f\in \C(\beta X)\cap M\}$. Hence the validity of assertion (ii) from Lemma~\ref{l:generskel} follows from
the previous lemma applied to $L=\overline{X\cap M}^{\beta X}$, $A=\overline{X\cap M}^X$ and $S=\{f\restriction_L:f\in C(\beta X)\cap M\}$.
The converse implication can be proved in the same way, just the final use of the previous lemma is not necessary.

 The equivalence of the respective assertions (i) is the content of the following proposition.

\begin{proposition}\label{p:rSkeletonInBetaX} Let $X$ be a countably compact space. Then $X$ has a full retractional skeleton if and only if it is induced by a retractional skeleton in $\beta X$.

Moreover,  if $\{r_s\}_{s\in\Gamma}$ is a full retractional skeleton in $X$, then there is a retractional skeleton $\{R_s\}_{s\in\Gamma}$ in $\beta X$ inducing $X$ such that $R_s\restriction_X = r_s$ for every $s\in\Gamma$.
\end{proposition}

\begin{proof} We start by proving the `if part'.  Let $\mathfrak{s} = \{R_s\}_{s\in\Gamma}$ be a retractional skeleton in $\beta X$ with $D(\mathfrak{s}) = X$. Then $\mathfrak{s'} = \{R_s\restriction_X\}_{s\in\Gamma}$ is a full retractional skeleton in $X$. Indeed, since $X$ is induced by $\mathfrak{s}$, we have, for each $s\in\Gamma$, $R_s[\beta X]\subset X$. Since $R_s$ is a retraction, we get $R_s[\beta X]=R_s[R_s[\beta X]]\subset R_s[X]$, hence $R_s[X]=R_s[\beta X]$. It follows that ranges of $R_s\restriction_X$ cover $X$. Further, it is immediate that (i)--(iv) from the definition of a retractional skeleton are satisfied. This finishes the proof.

To show the `only if' part let $\mathfrak{s} = \{r_s\}_{s\in\Gamma}$ be a full retractional skeleton in $X$. Fix $s\in\Gamma$. We extend the retraction $r_s:X\to\beta X$ to a continuous function $R_s:\beta X\to\beta X$. Then $R_s$ is a retraction because $R_s\circ R_s = R_s$ on a dense subset $X$; hence, $R_s\circ R_s = R_s$ on $\beta X$. Moreover, $r_s[X] = R_s[\beta X]$ because $R_s[X] = r_s[X]$ is compact and dense in $R_s[\beta X]$; 
hence, $R_s[\beta X] = r_s[X] \subset X$ is a metrizable compact. Now, it is immediate that $\mathfrak{s'} = \{R_s\}_{s\in\Gamma}$ is a system of retractions on $\beta X$ satisfying (i), (ii) from the definition of a retractional skeleton and $D(\mathfrak{s'}) = X$. In order to verify that (iii) from the definition of a retractional skeleton holds, let us fix a sequence $s_0 < s_1 < \cdots$ in $\Gamma$ with $t = \sup_{n\in\omega}s_n\in\Gamma$ and $x\in\beta X$. Then $R_t(x)\in X$. Therefore, $R_t(x) = r_t(R_t(x)) = \lim_{n\to\infty} r_{s_n}(R_t(x)) = \lim_{n\to\infty} R_{s_n}(x)$.

Finally, let us fix $x\in\beta X$. It remains to show that $\lim_{s\in\Gamma} R_s(x) = x$. Arguing by contradiction, let $y\in\beta X$ be a cluster point of the net $\{R_s(x)\}_{s\in\Gamma}$ with $y\neq x$. Fix $U\in\tau(x,\beta X)$ and $V\in\tau(y,\beta X)$ with $\ov{U}\cap\ov{V} = \emptyset$ and $s_0\in\Gamma$. We inductively find sequences $\{s_n\}_{n\in\en}$, $\{s'_n\}_{n\in\en}$ of indices from $\Gamma$, $\{x_n\}_{n\in\en}$ of points from $X$ and $\{U_n\}_{n\in\en}$ of sets from $\tau(x,\beta X)$ such that, for every $n\in\en$,

\begin{itemize}
	\item $\ov{U_1}\subset U$,
	\item $s_n\leq s'_n\leq s_{n+1}$, $\ov{U_{n+1}}\subset U_n$,
	\item $R_{s_n}(U_n)\subset V$,
	\item $x_n\in U_n\cap X$ and
	\item $R_{s'_n}(x_n) = x_n$.
\end{itemize}

Let us describe the inductive process. Let $s'_{n-1}$ and $U_{n-1}$ be defined (we put $s'_0 = s_0$ and $U_0 = U$ if $n=1$). Find $s_n\geq s'_{n-1}$ with $R_{s_n}(x)\in V$ and $U_n\in\tau(x,\beta X)$ with $\ov{U_n}\subset U_{n-1}$ and $R_{s_n}(U_n)\subset V$. Since  $X$ is dense in $\beta X$ we can find $x_n\in X\cap U_n$. Finally, we choose $s'_n$ to be such that $s'_n\geq s_n$ and $R_{s'_n}(x_n) = x_n$.

By Fact~\ref{f:basics}, $X$ is sequentially compact and so we may without loss of generality assume that the sequence $\{x_n\}_{n\in\en}$ converges to some $z\in X$. For $t = \sup_{n\in\omega}s_n$ we get
\[
R_t(z) = \lim_{n\to\infty} R_t(x_n) = \lim_{n\to\infty} R_t(R_{s'_n}(x_n)) = \lim_{n\to\infty} R_{s'_n}(x_n) = \lim_{n\to\infty} x_n = z.
\]
Hence, $R_t(z) = z\in\ov{U}$. Moreover, $x_k\in U_n$ for every $k\geq n$; hence, $z\in\bigcap_{n\in\en}\ov{U_n}$ and 
\[
R_t(z)\in R_t\left[\bigcap_{n\in\en} \ov{U_n}\right] = R_t\left[\bigcap_{n\in\en} U_n\right] \subset\ov{V},
\]
which is a contradiction with $\ov{U}\cap\ov{V} = \emptyset$.

\smallskip

The `moreover part' follows immediately from the construction.
\end{proof}

To complete the proof of Theorem~\ref{t:main} we will show (ii)$\Rightarrow$(iii). 
The idea is to use a Skolem function given by (ii) to create the required $\omega$-monotone mapping.
We begin with the following lemma which gives a formula for a network of a given retraction.

\begin{lemma}\label{l:networkForRetract}
Let $X$ be a space and $r:X\to X$ a retraction with a compact range. Let $S\subset\C(X)$ be a subset separating the points of $r[X]$ such that $f\circ r = f$ for every $f\in S$. Then 
\begin{multline*}\N(S) = \{f_1^{-1}(I_1)\cap\dots\cap f_k^{-1}(I_k)\setsep f_1,\dots,f_k\in S,\\ I_1,\dots, I_k\mbox{ are open intervals with rational endpoints}\}\end{multline*} is a network of $r$.
\end{lemma}

\begin{proof}Fix $y\in X$ and $U\in\tau(r(y),r(X))$. For any  $z\in r[X]\setminus U$ there is a function $f_z\in S$ such that $f_z(y)\ne f_z(z)$.
We can find disjoint open intervals $I_z$ and $J_z$ with rational endpoints such that $f_z(y)\in I_z$ and $f_z(z)\in J_z$. The open sets $f_z^{-1}(J_z)$, $z\in r[X]\setminus U$ cover the compact set $r[X]\setminus U$, so there are $z_1,\dots,z_k\in r[X]\setminus U$ such that $\bigcup_{i=1}^k f_{z_i}^{-1}(J_{z_i})\supset  r[X]\setminus U$. Then
\[
y\in \bigcap_{i=1}^k f_{z_i}^{-1}(I_{z_i})\subset r^{-1}[U],
\]
which completes the proof.
\end{proof}

We continue by the following lemma which we use to improve a bit the model provided by the assumption (ii).

\begin{lemma}\label{l:retrakce} Let $X$ be a countably compact space satisfying the assumption (ii) from 
Theorem~\ref{t:main}. Then for any suitable model $M$ the following holds:
\begin{itemize}
	\item[(i)] $\C(X)\cap M$ separates the points of $\overline{X\cap M}$.
	\item[(ii)] $\overline{X\cap M}$ is compact.
	\item[(iii)] There is a retraction $r:X\to X$ with $r[X]=\ov{X\cap M}$  such that $f\circ r= f$ for each
	$f\in\C(X)\cap M$. 
\end{itemize}
\end{lemma}

\begin{proof} 
In this proof we will use the identification of any $n\in\omega$ with the set $\{0,\dots,n-1\}$.
Further, denote by $\B$ the set of all the open intervals with rational endpoints and by $\B^{<\omega}$ the set of all the functions whose domain is some $n\in\omega$ and whose values are in $\B$.

Fix formulas $\phi_1,\ldots,\phi_n$ containing all the formulas $\theta_1,\dots,\theta_m$ from Lemma \ref{l:predp}, the formulas provided by the assertion (ii) from Theorem~\ref{t:main} and the formulas (and their subformulas) marked by $(*)$ in the proof below, and a countable set $Y$ containing the set $Y_0$ from Lemma \ref{l:predp}, the set provided by the assertion (ii) from Theorem~\ref{t:main} and the set $\{X, \C(X), \B, \B^{<\omega}\}$. Fix $M\prec(\phi_1,...,\phi_n;\; Y)$.

Set $A=\C(X)\cap M$. By the assumptions $A$ separates points of $\ov{X\cap M}$, hence the assertion (i) is fulfilled. Let us define
the mapping $\Phi:X\to \er^A$ by the formula
\[
\Phi(x)(f)=f(x), \quad f\in A,x\in X.
\]
Then $\Phi$ is continuous, hence $\Phi[X]$ is countably compact. Since $A$ is countable, we deduce that $\Phi[X]$ is a metrizable compact.
Moreover, $\Phi$ is a closed mapping (any closed $F\subset X$ is countably compact, $\Phi[F]$ is then countably compact, hence compact and thus closed).
By the already proved condition (i) the mapping $\Phi$ is one-to-one when restricted to $\ov{X\cap M}$, so it is a homeomorphism of $\ov{X\cap M}$
onto its image. In particular, $\ov{X\cap M}$ is compact, which proves the assertion (ii).

The next step is to prove that $\Phi[X]=\Phi[\ov{X\cap M}]$. Since $\Phi$ is a closed mapping, it is enough to show that $\Phi[X\cap M]$ is
dense in $\Phi[X]$. To do that fix $x_0\in X$ and $U$ an open set in $\er^A$ containing $\Phi(x_0)$. It follows that there is a finite set
$F\subset A$ and intervals $I_f\in \B$, $f\in F$, such that
\[
\Phi(x_0)\in \{z\in \er^A\setsep z(f)\in I_f\mbox{ for }f\in F\}\subset U,
\]
which means
\[
x_0\in\{x\in X\setsep f(x)\in I_f\mbox{ for }f\in F\}\subset \Phi^{-1}(U).
\]
Since $F\subset A\subset M$ and $F$ is finite, Lemma~\ref{l:predp} yields $F\in M$. Further, by absoluteness of the formula
\[
\exists n\in\omega\; \exists \eta\quad (\eta \mbox{ is a mapping of $n$ onto }F)\eqno{(*)}
\]
and its subformulas, there is $n\in\omega$ and an onto mapping $\eta:n\to F$ in $M$. Let us further define mapping $\zeta:n\to\B$ by
$\zeta(i)=I_{\eta(i)}$. Since $\zeta\in\B^{<\omega}$, $\B^{<\omega}\in M$ and it is countable, it follows from Lemma~\ref{l:predp} that $\zeta\in M$.
Finally, by absoluteness of the formula
\[
\exists x\in X\; \forall i\in n \quad(\eta(i)(x)\in \zeta(i)) \eqno{(*)}
\]
and its subformulas, there is $x\in X\cap M$ such that $\eta(i)(x)\in\zeta(i)$ for each $i\in n$, in other words $f(x)\in I_f$ for each $f\in F$,
hence $\Phi(x)\in U$. This completes the proof that $\Phi[X\cap M]$ is dense in $\Phi[X]$, hence $\Phi[X]=\Phi[\ov{X\cap M}]$.

Finally, set $r=\left(\Phi\restriction_{\ov{X\cap M}}\right)^{-1}\circ \Phi$. It is clear that $r$ is a continuous retraction on $X$ with the range $\ov{X\cap M}$. Further, if $f\in A$ and $x\in X$, then
\[
(f\circ r)(x)=f(r(x))=\Phi(r(x))(f)=\Phi\left(\left(\Phi\restriction_{\ov{X\cap M}}\right)^{-1}\left( \Phi(x)\right)\right)(f)=\Phi(x)(f)=f(x),
\]
hence $f\circ r=f$, which completes the proof of (iii).
\end{proof}

Finally, we give the proof of the missing implication of Theorem~\ref{t:main}.

\smallskip

Let us assume that the assertion (ii) holds.  We will show that (iii) holds as well.  
For any countable $S\subset \C(X)$ let us define $\N(S)$ by the formula given in  Lemma~\ref{l:networkForRetract}. 
Note that the assignment $S\mapsto \N(S)$ is $\omega$-monotone. 
Let a subformula closed list of formulas $\phi_1,\ldots,\phi_n$ and a countable set $Y$ be the ones provided by Lemma~\ref{l:retrakce}. 
By Theorem~\ref{T:countable-model} there exists a set $R\supset X\cup Y$ such that $\phi_1,\ldots,\phi_n$ are absolute for $R$.
Let $\psi$ be a Skolem function for $\phi_1,\ldots,\phi_n$ and $R$; see Lemma~\ref{l:BasicSkolem}. For every $A\in[X]^{\leq\omega}$, we put $M(A) = \psi(A\cup Y)$. By Lemma~\ref{l:BasicSkolem}, $M(A)\prec(\phi_1,\ldots,\phi_n;Y)$, $M(A)\supset A$ and the assignment $A\mapsto M(A)$ is $\omega$-monotone. Let $r_A$ be the retraction assigned to $M(A)$ by Lemma~\ref{l:retrakce} and $\O(A) = \N(\C(X)\cap M(A))$. Then $\O(A)$ is a countable network of $r_A$ and $A\subset M(A)\cap X\subset \ov{M(A)\cap X} = r_A[X]$. Finally, the assignment $A\mapsto \O(A)$ is $\omega$-monotone because it is a composition of $\omega$-monotone mappings $A\mapsto \C(X)\cap M(A)$ and $\N$. This completes the proof.

\section{A function-space characterization of compact spaces with retractional skeleton}

The aim of this section is to prove Theorem~\ref{t:answerCuth} and Theorem~\ref{t:answerKubis}. In fact, instead of the latter we prove a more precise version, namely Theorem~\ref{t:skeletonIffXSokolov} below. 

Let us start with Theorem~\ref{t:answerCuth}. It answers \cite[Problem 1]{cuthCMUC} and can be viewed as a noncommutative analogue of \cite[Theorem 2.1]{kalendaCharact}. Let us comment it a bit. Let $K$ be a compact space and $D\subset K$ dense subset. Then $D$ is induced by a commutative retractional skeleton in $K$ (here, \emph{commutative} means that each two projections from the skeleton commute, not only the compatible pairs) 
if and only if $D$ is a ``$\Sigma$-subset'' of $K$; see e.g. the \cite[p. 56]{cuthCMUC}. In \cite[Theorem 2.1]{kalendaCharact} it is proved that $D$ is a $\Sigma$-subset of $K$ if and only if $D$ is countably compact and $(\C(K),\tau_p(D))$ is primarily Lindel\"of (see {\cite[Definition 1.2]{kalendaCharact}}). To characterize sets induced by a possibly noncommutative retractional skeleton, we replace the property to be primarily Lindel\"of by the monotonical Sokolov property.

\begin{proof}[Proof of Theorem \ref{t:answerCuth}]
Assume that $D$ is induced by a retractional skeleton in $K$. One can notice that, by Fact~\ref{f:basics}, $D$ is countably compact and $K=\beta D$.
Hence, $D$ has a full retractional skeleton (by the trivial implication of Proposition~\ref{p:rSkeletonInBetaX}). By Theorem~\ref{t:main}, $D$ is monotonically retractable; hence, by Fact~\ref{f:mrs}, $\C_p(D)$ is monotonically Sokolov. Taking into account that $\beta D = K$, $\C_p(D)$ is homeomorphic to $(\C(K),\tau_p(D))$. Thus, $(\C(K),\tau_p(D))$ is monotonically Sokolov.

For the converse implication, let us assume that $D$ is countably compact and $(\C(K),\tau_p(D))$ is monotonically Sokolov. 
By Fact~\ref{f:mrs} the space $\C_p(\C(K),\tau_p(D))$ is monotonically retractable. The mapping $\Psi:D\to\C_p(\C(K),\tau_p(D))$
defined by
\[
\Psi(d)(f)=f(d),\quad f\in \C(K), d\in D
\]
is continuous (by the very definition of the respective topologies) and one-to-one. Further, for any closed
$F\subset D$ its image $\Psi(F)$ is countably compact and hence closed in $\C_p(\C(K),\tau_p(D))$ (by Fact~\ref{f:mrs}). It follows that
$\Psi$ is a homeomorphism of $D$ onto a closed subset of $\C_p(\C(K),\tau_p(D))$, thus $D$ is monotonically retractable. So, by Theorem~\ref{t:main} that $D$ has a full retractional skeleton. It follows that $D$ is countably closed in $K$ (it follows from the definitions that the closure in $D$ of any countable subset of $D$ is compact). Further, $(\C(K),\tau_p(D))$ is Lindel\"of by Fact~\ref{f:mrs}. Hence, by \cite[Proposition 2.13]{kalendaCharact}, $\beta D = K$.  By Proposition~\ref{p:rSkeletonInBetaX}, $D$ is induced by a retractional skeleton in $K$.\end{proof}

Let us now formulate the more precise version of Theorem~\ref{t:answerKubis} which we will prove. It is the following theorem which can be viewed 
as a noncommutative version of \cite[Theorem 2.7]{kalendaSurvey} (which is a precise formulation of \cite[Theorem 2.3]{kalendaCharact}).

\begin{theorem}\label{t:skeletonIffXSokolov} Let $E$ be a Banach space and $D$ a dense subset of $(B_{E^*},w^*)$. Then the following are equivalent:
\begin{enumerate}[\upshape (i)]
	\item $\sspan(D)$ is induced by a 1-projectional skeleton in $E$ and $\sspan(D)\cap B_{E^*} = D$.
	\item $D$ is a convex symmetric set induced by a retractional skeleton in $(B_{E^*},w^*)$.
	\item $D$ is weak$^*$ countably compact and $(E,\sigma(E,D))$ is monotonically Sokolov.
\end{enumerate}
\end{theorem}

A \emph{projectional skeleton} in a Banach space $E$ is an indexed system of bounded linear projections on $E$ with the properties of a retractional skeleton, except for the first property which is replaced by the assumption that the projections have separable ranges. By metrizability the last condition implies that the ranges of projections cover $E$. A $1$-projectional skeleton is a projectional skeleton formed by norm one projections.
The subspace generated by a projectional skeleton is the union of ranges of adjoint projections. The adjoint projections of a $1$-projectional skeleton on $E$ form a retractional skeleton on $(B_{E^*},w^*)$. For exact definitions and explanations see \cite{kubisSkeleton} or \cite{cuthSimul}.
We will deal with projectional skeletons via retractional skeletons using the following lemma. In its proof we use the notation from \cite{kubisSkeleton}.

\begin{lemma}\label{l:projectionalIffRetractional}Let $E$ be a Banach space and $D\subset E^*$ a 1-norming subspace of $E^*$. Then $D$ is induced by a 1-projectional skeleton if and only if $D\cap B_{E^*}$ is induced by a retractional skeleton in $(B_{E^*},w^*)$.
\end{lemma}
\begin{proof} The `only if part' is easy and is proved in \cite[Theorem 4.2]{cuthSimul}. Let us prove the `if part'. Suppose that
$D\cap B_{E^*}$ is induced by a retractional skeleton in $(B_{E^*},w^*)$. By \cite[Theorem 4.2]{cuthSimul}, $D$ is a subset of a set $D(\mathfrak{s})$ induced by a 1-projectional skeleton. Thus, by the `if part', $D(\mathfrak{s})\cap B_{X^*}$ is induced by a retractional skeleton. We have $D\cap B_{X^*}\subset D(\mathfrak{s})\cap B_{X^*}$, both induced by a retractional skeleton. Consequently, by \cite[Lemma 3.2]{cuthCMUC}, $D\cap B_{X^*} = D(\mathfrak{s})\cap B_{X^*}$ and $D = D(\mathfrak{s})$ is induced by a 1-projectional skeleton.
\end{proof}

Now we proceed to the proof of  Theorem~\ref{t:skeletonIffXSokolov}. We follow the line of the proof of \cite[Theorem 2.3]{kalendaCharact}; instead of ``$\Sigma$-subset'' we use the notion of a set induced by a retractional skeleton and instead of ``homeomorphic to a closed coordinatewise bounded subset of some $\Sigma(\Gamma)$'' we use spaces with a full retractional skeleton. Thus, some technical details must be handled in a slightly different way. Namely, we need the following analogue of \cite[Lemma 2.18]{kalendaCharact}.

\begin{lemma}\label{l:kalendaBetaA}Let $X$ be a Banach space and $A\subset (B_{E^*},w^*)$ be a dense, convex and symmetric set with a full retractional skeleton. If $(E,\sigma(E,A))$ is Lindel\"of, then $B_{E^*} = \beta A$.
\end{lemma}
\begin{proof}
The proof is identical with the proof of \cite[Lemma 2.18]{kalendaCharact} which goes through a technical \cite[Lemma 2.17]{kalendaCharact}. 
There, instead of a set with a full retractional skeleton (resp. a set induced by a retractional skeleton), set ``homeomorphic to a closed coordinatewise bounded subset of some $\Sigma(\Gamma)$'' (resp. ``$\Sigma$-subset'') is considered. Thus, it is enough to observe that sets with a full retractional skeleton (resp. sets induced by a retractional skeleton) have the topological properties needed in the proofs. 

Namely, it is enough to use the following properties:
\begin{itemize}
	\item If $D$ has a full retractional skeleton, is is countably closed in each superspace, in particular $A$ is countably closed in $(B_{E^*},w^*)$.
	\item If $D$ has a full retractional skeleton, it is induced by a retractional skeleton in $\beta D$ (Proposition~\ref{p:rSkeletonInBetaX}).
	\item If $D$ is induced by a retractional skeleton in a compact space $K$ and $F\subset D$ is relatively closed, then $F$ is induced by a retractional skeleton in $\ov{F}^K$ (\cite[Lemma 3.5]{cuthCMUC}) and hence $\ov{F}^K=\beta F$ (Fact~\ref{f:basics}).
	\item If $D_i$ are sets induced by a retractional skeleton in compact spaces $K_i$ for $i = 1,\ldots,n$, then $D_1\times\ldots\times D_n$ is induced by a retractional skeleton in $K_1\times\ldots\times K_n$ (see the proof of \cite[Theorem 31]{kubisSkeleton}). 
\end{itemize}
\end{proof}

\begin{proof}[Proof of Theorem \ref{t:skeletonIffXSokolov}] Implication (i)$\Rightarrow$(ii) follows immediately from Lemma~\ref{l:projectionalIffRetractional}. Let us assume that (ii) holds. By Fact~\ref{f:basics}, $D$ is weak$^*$ countably compact. Put $K = (B_{E^*},w^*)$. By Theorem~\ref{t:answerCuth}, $(\C(K),\tau_p(D))$ is monotonically Sokolov. Consider the $\sigma(E,D)$-$\tau_p(D)$ homeomorphism $I:E\to\C(B_{E^*},w^*)$ defined by $I(x)(x^*) = x^*(x)$, $x\in E$, $x^*\in E^*$. It is a standard fact, see e. g. \cite[Lemma 4.4]{cuthSimul}, that $I(E)$ is $\tau_p(D)$-closed in $\C(B_{E^*},w^*)$. By Fact~\ref{f:mrs} $I(E)$ is monotonically Sokolov and hence $(E,\sigma(E,D))$ is monotonically Sokolov.

It remains to prove (iii)$\Rightarrow$(i). Let us assume that (iii) holds. Since $(E,\sigma(E,D))$ is monotonically Sokolov, $\C_p(E,\sigma(E,D))$ is monotonically retractable due to Fact~\ref{f:mrs}. Observe that $\sspan(D)\subset \C(E,\sigma(E,D))$ and that the inclusion map $i:(\sspan(D),w^*)\to \C_p(E,\sigma(E,D))$ is a homeomorphism. Since $D$ is weak$^*$ countably compact, it is a countably compact subset of $\C_p(E,\sigma(E,D))$,
so $D$ is closed in the latter space by Fact~\ref{f:mrs}. Hence $D$ is monotonically retractable. If we put $A = \sspan(D)\cap B_{E^*}\subset\C_p(E,\sigma(E,D))$, then $D$ is dense in $A$ (since it is weak$^*$ dense in $B_{E^*}$). It follows that $D=A$.
By Theorem~\ref{t:main} $A$ has a full retractional skeleton. By Lemma~\ref{l:kalendaBetaA}, $\beta A = B_{E^*}$ and hence $A$
is induced by a retractional skeleton in $B_{E^*}$ by Proposition~\ref{p:rSkeletonInBetaX}. Finally, by Lemma~\ref{l:projectionalIffRetractional}, $\sspan(D)$ is induced by a 1-projectional skeleton in $E$.
\end{proof}

We finish by showing how Theorem~\ref{t:answerKubis} follows from Theorem~\ref{t:skeletonIffXSokolov}. Similarly as above, for details concerning projectional skeletons we refer to \cite{kubisSkeleton} where all the fact needed in the proof bellow may be found.

\begin{proof}[Proof of Theorem \ref{t:answerKubis}]
Let $\langle E,\|\cdot\|\rangle$ be a Banach space and $D\subset E^*$ a norming subspace. Then there is an equivalent norm $|\cdot|$ on $X$ such that $D$ becomes 1-norming; see e. g. \cite[Proposition 1]{kubisSkeleton}. Then $D\cap B_{\langle E,|\cdot|\rangle^*}$ is weak$^*$ dense in $B_{\langle E,|\cdot|\rangle^*}$. So,
it follows from Theorem~\ref{t:skeletonIffXSokolov} that $D$ is induced by a $1$-projectional skeleton in $E$ if and only $D\cap B_{\langle E,|\cdot|\rangle^*}$
is weak$^*$-countably compact and $(E,\sigma(E,D))$ is monotonically Sokolov. Since the topology $\sigma(E,D)$ does not depend on the choice of an equivalent norm and any subspace induced by a projectional skeleton is weak$^*$ countably closed, the assertion (ii) holds if and only if $D$ is induced by a $1$-projectional skeleton in $\langle E,|\cdot|\rangle$. 

Now, it is enough to show that $D$ is induced by a projectional skeleton in $\langle E,\|\cdot\|\rangle$ if and only if $D$ is induced by a 1-projectional skeleton in $\langle E,|\cdot|\rangle$. If $D$ is induced by a 1-projectional skeleton in $\langle E,|\cdot|\rangle$,
the same system of projections is a projectional skeleton in $\langle E,\|\cdot\|\rangle$ and induces $D$.
Conversely, let $D$ be induced by a projectional skeleton in $\langle E,\|\cdot\|\rangle$.  By \cite[Theorem 15]{kubisSkeleton} then $D$ generates projections in $E$ and there exists a 1-projectional skeleton $\mathfrak{s}$ in $\langle E,|\cdot|\rangle$ such that $D\subset D(\mathfrak{s})$. 
Since the projectional skeleton in $\langle E,\|\cdot\|\rangle$ inducing $D$ remains to be a projectional skeleton in $\langle E,|\cdot|\rangle$,
\cite[Corollary 19]{kubisSkeleton} implies $D= D(\mathfrak{s})$.
\end{proof}

\def\cprime{$'$}

\end{document}